\def\FULLVERSION{}

\documentclass{elsarticle}
\ifdefined\FULLVERSION
\journal{somewhere}
\else
\journal{Journal of Mathematical Analysis and Applications}
\fi

\biboptions{sort}

\usepackage{geometry}
\usepackage{array}
\usepackage{setspace}

\usepackage{enumitem}		
\usepackage[shortcuts]{extdash}

\usepackage{nohyperref}  
\usepackage{url}         

\usepackage{amsmath,amsfonts,amssymb,amstext,amsthm,bbm} 

\usepackage{version}

\newcommand{\range}[2]{\in\{#1,\dots,#2\}}

\newtheorem{theorem}{Theorem}
\newtheorem*{theorem*}{Theorem}
\newtheorem*{conjecture*}{Conjecture}

\begin{document}

\begin{frontmatter}{}

\title{Cake-Cutting with Different Entitlements: How Many Cuts are Needed?}
\author{Erel Segal-Halevi}
\address{Ariel University, Ariel 40700, Israel. erelsgl@gmail.com}
\begin{abstract}
A cake has to be divided fairly among $n$ agents. When all agents have equal entitlements, it is known that such a division can be implemented with $n-1$ cuts. 
When agents may have different entitlements, the paper shows that at least $2 n -2$ cuts may be necessary, and $O(n\log n)$ cuts are always sufficient.

\end{abstract}

\end{frontmatter}{}

\section{Introduction}
The \emph{fair cake-cutting} problem involves a continuous resource dubbed ``cake'' that has to be divided among several agents with different value measures on the cake. 
The problem was introduced by \citet{Steinhaus1948Problem} and inspired hundreds of research papers and several books \citep{Brams1996Fair,Robertson1998CakeCutting,Moulin2004Fair,Brams2007Mathematics}. 
See 
\ifdefined\FULLVERSION
\cite{Procaccia2015Cake,Branzei2015Computational,Lindner2016,Brams2017FairDivision,SegalHalevi2017Phd}
\else
\cite{Procaccia2015Cake,Branzei2015Computational,Lindner2016,Brams2017FairDivision}
\fi
for recent surveys.

An important requirement in dividing a cake is to minimize the number of \emph{cuts} used to implement the division%
\ifdefined\FULLVERSION
(see \cite{segalhalevi2018archipelago} and the references therein)%
\fi
.
When a cake is divided among $n$ people, it would be ideal to make $n-1$ cuts, such that each person receives a single connected piece. Making too many cuts might result in each person receiving a bunch of crumbs. Connectivity is even more important when the ``cake'' represents a resource such as land or time. When a land-estate is divided among heirs, a plot made of many disconnected pieces cannot be efficiently built upon or cultivated. Similarly, when the usage-time of a vacation-house is divided among its owners, having many disconnected time-intervals is practically useless.

In the simplest variant of the cake-cutting problem, all agents have equal entitlements. 
The common fairness criterion in this case is that each agent should receive a piece that is worth for him/her at least $1/n$ of the total cake value (a requirement called \emph{proportionality} or \emph{fair-share guarantee}). \citet{Steinhaus1948Problem} already proved that a proportional cake division can be implemented using $n-1$ cuts, which is obviously the smallest number of cuts needed for making $n$ pieces.

However, in many division problems, agents have different entitlements. As an example, consider a land-estate owned by a partnership in which Alice's share is 2/5 and George's share is 3/5. If the partnership is dissolved, then Alice is entitled to 2/5 and George is entitled to 3/5 of the land-estate. As another example, consider a vacation-house bought by several friends, each of whom invested a different sum of money towards the purchase. The time in which the house is used should be divided among the friends in proportion to their investment. 

In general, in the \emph{cake-cutting with entitlements} problem, each agent $i$ has an entitlement $t_i$ such that the sum of all entitlements is $1$. A division is called \emph{$t$-proportional} if each agent $i$ values his/her share as at least a fraction $t_i$ of the total cake value. Proportionality is the special case of $t$-proportionality in which $t_i=1/n$ for every agent $i$. The present paper studies the following question:

\begin{quote}
\emph{How many cuts are required to implement a $t$-proportional allocation among $n$ agents?}
\end{quote}

A first solution that comes to mind is  \emph{agent cloning}. In the example above, create 2 clones of Alice and 3 clones of George and find a proportional division among the clones. Since each clone receives at least $1/5$ of the total cake value, each original agent receives at least his/her entitlement. However, this solution might result in a large number of cuts. In particular, if the common denominator of all agents' entitlements is $D$ then agent-cloning might require up to $D-1$ cuts. 

Several authors have tried to improve on the agent-cloning method with respect to \emph{computational complexity} --- the number of evaluation queries and intermediate cut-marks required to find a proportional division. See
\citet{McAvaney1992Ramsey}, 
\citet{barbanel1996game},
\citet{Robertson1997Extensions}, \citet[pages 36-–44]{Robertson1998CakeCutting},
\citet{Brams2011DivideandConquer}.
The most recent advancement in this front is by \citet{cseh2018complexity}. After summarizing the computational complexity of previous algorithms, they present a faster one that requires $2(n-1)\lceil\log_2{D}\rceil$ queries, where $D$ is the smallest common denominator of the entitlements. They prove that this number is asymptotically optimal. 

In contrast, the present paper focuses on the question of how many cuts are required to implement the \emph{final} division, regardless of how this division is computed.
I am aware of two previous studies of this question:  \citet{Brams2011DivideandConquer} prove that $n-1$ cuts may be insufficient, and \citet{Shishido1999MarkChooseCut} prove that $2 \cdot 3^{n-1}$ cuts are always sufficient, and claim in their concluding remarks that the number can be improved to $O(n^2)$ cuts.
This paper shows two new bounds for this question.
\begin{itemize}
\item Lower bound: $2 n -2$ cuts might be necessary (Section \ref{sec:lower}).
\item Upper bound: 
$2 n \log_2{\widehat{n}} - 2 \widehat{n} + 2$ cuts are always sufficient, where $\widehat{n} := 2^{\lceil\log_2{n}\rceil}=n$ rounded up to the nearest power of two (Section \ref{sec:upper}).
\end{itemize}
The bounds coincide when $n=2$ at 2 cuts,
but they diverge when $n=3,4,5...$, where the lower bound is $4,6,8...$ and the upper bound is $6,10,16...$.%
\ifdefined\FULLVERSION
\footnote{
The situation of agents with different entitlements has been studied also for allocation of indivisible items
\citep{Babaioff2017Competitive,Farhadi2017Fair,SegalHalevi2018CEFAI}.
}
\fi

\section{Definitions}
There is a cake $C$, represented by an interval (a connected subset of $\mathbb{R}$).
There is a set $N$ of $n$ agents, represented by $n$ nonatomic value-measures $(V_i)_{i\in N}$ on $C$.
There is a vector $t$ of length $n$ where each element $t_i$ represents the entitlement of agent $i$. The vector is normalized such that $\sum_{i\in N} t_i =1$.

The goal is to find a partition of $C$ among the agents, $(X_i)_{i\in N}$, where the $X_i$ are pairwise-disjoint and $\cup_{i\in N} X_i = C$. The partition should be \emph{$t$-proportional}, i.e, for every agent $i\in N$:
\begin{align*}
V_i(X_i) \geq t_i \cdot V_i(C).
\end{align*}

\section{Lower Bound}
\label{sec:lower}
\begin{theorem}
\label{thm:lower}
For every $n\geq 2$, there exists a set of $n$ value measures and an entitlement-vector $t$ such that at least $2 n-2$ cuts are required for a $t$-proportional division.
\end{theorem}
\begin{proof}
Suppose the entitlement of agent 1 is ${n-0.9\over n}$ and the entitlement of each of the other agents is ${0.9 \over n(n-1)}$. The cake is an interval consisting of $2 n-1$ subintervals and the agent valuations in these subintervals are:
\begin{center}
\begin{tabular}{|c|c|c|c|c|c|c|c|c|c|c|c|}
	\hline
	\hline Agent 1 & 1 & 0 & 1 & 0 & 1 & 0 & 1 & ... & 1 & 0 & 1
	\\
	\hline Agent 2 & 0 & 1 & 0 & 0 & 0 & 0 & 0 & ... & 0 & 0 & 0
	\\
	\hline Agent 3 & 0 & 0 & 0 & 1 & 0 & 0 & 0 & ... & 0 & 0 & 0
	\\
	\hline Agent 4 & 0 & 0 & 0 & 0 & 0 & 1 & 0 & ... & 0 & 0 & 0
	\\
	\hline ... &  &  &  &  &  &  &  & ... &  &  & 
	\\
	\hline Agent $n$ & 0 & 0 & 0 & 0 & 0 & 0 & 0 & ... & 0 & 1 & 0
	\\
	\hline
\end{tabular}
\end{center}
Agent 1 must receive more than ${n-1\over n}$ of his cake value, so he must receive a positive slice of each of his $n$ positive subintervals. But, he cannot receive a single interval that touches two of his positive subintervals, since such an interval will leave one of the other agents with zero value. Therefore, agent 1 must receive at least $n$ different intervals. Each of the other agents must receive at least one interval, so the total number of intervals is at least $2 n-1$. Therefore at least $2 n-2$ cuts are needed.
\end{proof}

\section{Upper Bound}
\label{sec:upper}
The upper bound uses the \emph{Stromquist-Woodall theorem} \citep{Stromquist1985Sets}. They consider a cake in which the left and right endpoints of the cake are identified (In other words, $C$ is a 1-dimensional ``pie'').
\begin{theorem*}[\citet{Stromquist1985Sets}]
Consider a 1-dimensional pie $C$, 
a set $N$ of $n$ agents with nonatomic measures on $C$, and a number $r\in[0,1]$.
There always exists a subset $C_r \subseteq C$, which is a union of at most $n-1$ intervals, such that for every $i\in N$:
\begin{align*}
V_i (C_r) = r\cdot V_i(C).
\end{align*}
The number $n-1$ is tight when the denominator in the reduced fraction of $r$ is at least $n$
(in particular, it is tight when $r$ is irrational).
\end{theorem*}
For the theorem below, denote $\widehat{n} := n$ rounded up to the nearest power of two $= 2^{\lceil\log_2{n}\rceil}$.
\begin{theorem}
\label{thm:upper}
For every $n$, set of $n$ value measures, and an entitlement-vector $t$,
a $t$-proportional allocation can be implemented using at most
$2 n \log_2{\widehat{n}} - 2 \widehat{n} + 2$ cuts.
\end{theorem}
\begin{proof}
In this proof, for convenience, the valuations are normalized such that the total cake value is $1$.

The proof is by induction on $n$. For $n=1$, the single agent gets the entire cake with zero cuts.
Suppose the claim is true for all integers smaller than $n$.
Given $n$ agents, let $n_A := \lfloor n/2 \rfloor$ and $n_B := n - n_A = \lceil n/2 \rceil$. Partition the set of agents $N$ into two disjoint subsets: $N_A$ with $n_A$ agents and $N_B$ with $n_B$ agents. Define:
\begin{align*}
r_A := \sum_{i\in N_A} t_i && r_B := 1-r_A = \sum_{i\in N_B} t_i 
\end{align*}
By the Stromquist-Woodall theorem, there exists a piece $C_A$ that all $n$ agents value at exactly $r_A$. Define $C_B := C\setminus C_A$ and note that all agents value $C_B$ at exactly $r_B$.
If the two endpoints of $C$ are identified, then $C_A$ is a union of at most $n-1$ intervals. Hence, it is possible to separate $C_A$ from $C_B$ using at most $2 n - 2$ cuts --- two cuts per interval.
The same cuts can be used when the endpoints are not identified (even though in this case $C_A$ might be a union of $n$ intervals).

For each $k\in\{A,B\}$, define a new entitlement\-/vector $t_k := (t_i / r_k)_{i\in N_k}$.
By the induction assumption, 
there exists a $t_k$\-/proportional allocation of $C_k$ among the agents in $N_k$ using at most 
$2 n_k \log_2{\widehat{n_k}} - 2 \widehat{n_k} + 2$ cuts.
Note that 
$\widehat{n_k} \leq \widehat{n}/2$,
so the number of cuts required for each of the two sub-divisions is at most
$2 n_k (\log_2{\widehat{n}}-1) -  \widehat{n} + 2$.

Combining the two allocations yields a $t$-proportional allocation of $C$. The total number of required cuts is at most:
\begin{align*}
&
[2n-2]
+
[2 n_A (\log_2{\widehat{n}}-1) -  \widehat{n} + 2]
+
[2 n_B (\log_2{\widehat{n}}-1) -  \widehat{n} + 2]
\\
=&
[2n-2]
+
[2 n (\log_2{\widehat{n}}-1) -  2\widehat{n} + 4]
\\
=&
2 n \log_2{\widehat{n}} -  2\widehat{n} + 2.
\qedhere
\end{align*}
\end{proof}
For $n=2$, both the lower and the upper bound equal $2$, but for $n=3$ they differ --- the lower bound is $4$ and the upper bound is $6$. The theorems below show some cases in which the lower bound is attainable.
\begin{theorem}
With $n=3$ agents, a $t$-proportional allocation with at most $4$ cuts exists when:

(a) $t_i = 1/2$ for some $i\in\{1,2,3\}$, or ---

(b) $t_i = t_j$ for some $\{i,j\}\subseteq \{1,2,3\}$, and the entitlements are rational.
\end{theorem}
\begin{proof}
(a) W.l.o.g. assume that $t_1 = 1/2$. Consider the subset of agents $\{2,3\}$ and the entitlement\-/vector $(2 t_2, 2 t_3)$. By Theorem \ref{thm:upper}, using two cuts it is possible to give each agent $i\in\{2,3\}$ a piece worth at least $2 t_i$. Now each agent $i\in\{2,3\}$ can divide his/her piece with agent 1 using cut-and-choose. This requires two additional cuts, and results in a $t$-proportional allocation of $C$.

(b)  W.l.o.g. assume that, for some integers $B$ and $D$, $t_1 = t_2 = B/D$. So $t_3 = (D-2B)/D$. 
Normalize the valuations such that the total cake value is $D$.
Assume that the cake is a circle by identifying its two endpoints. 
Ask agent 1 to mark a partition of $C$ into $D$ parts, such that each part is worth for him exactly $1$.
Ask agent 3 to select $B$ consecutive parts with a smallest value for her. Denote this sequence of $B$ parts by $X$.
By the pigeonhole principle, $V_3(X)\leq B$. Now there are two cases:
\begin{itemize}
\item $V_2(X)\leq B$ too. Then, agent 1 cuts $X$ from the cake using two cuts, and takes it. Agent 1's value is exactly $B$. The remainder is worth for agents 2 and 3 at least $D-B$; they share it among them in entitlement-ratios ${B\over D-B},{D-2B\over D-B}$. By Theorem \ref{thm:upper}, this requires two additional cuts.
\item $V_2(X)\geq B$. Then, agent 2 cuts $X$ from the cake using two cuts, and takes it. Agent 2's value is at least $B$. The remainder is worth for agent 1 exactly $D-B$ and for agent 3 at least $D-B$; they share it among them in entitlement-ratios ${B\over D-B},{D-2B\over D-B}$. This requires two additional cuts.\qedhere
\end{itemize}
\end{proof}

\begin{theorem}
For every $n\geq 2$, a $t$-proportional allocation with at most $2 n - 2$ cuts exists if, for some integer $D\geq 1$, at least $n-1$ entitlements in $t$ are equal to $1/D$.
\end{theorem}
\begin{proof}~
Assume t.l.o.g. that, for some integer $D$, $t_1= \cdots =t_{n-1} = 1/D$. So $t_n = (D-n+1)/D$.
Normalize the valuations such that the total cake value is $D$.
Create $D-n+1$ clones of agent $n$.
Find a connected proportional division of the cake among all $D$ agents. 
In this division, each agent $i\range{1}{n-1}$ receives a single connected piece worth at least $1$; cut it from the cake using $2$ cuts. 
The remaining cake is worth for agent $n$ at least $D-n+1$.
All in all, $2(n-1)$ cuts are required.
\end{proof}

Based on these special cases, it may be conjectured that the lower bound is always attainable:
\begin{conjecture*}
For every $n$, set of $n$ value measures, and an entitlement-vector $t$,
there exists a $t$-proportional allocation with at most
$ 2 n - 2$ cuts.
\end{conjecture*}
The simplest case in which the conjecture is open is $n=3$ agents with entitlements  $(1/7,~2/7,~4/7)$.

\section{Acknowledgments}
I am grateful for helpful discussions with Sam Zbarsky%
\footnote{in https://mathoverflow.net/q/242112/34461}
and Alex Ravsky.%
\footnote{in https://math.stackexchange.com/q/3039348/297801}

\section*{References}
\bibliographystyle{elsarticle-num-names-alpha}
\bibliography{../erelsegal-halevi}

\begin{thebibliography}{21}
\providecommand{\natexlab}[1]{#1}
\providecommand{\url}[1]{\texttt{#1}}
\providecommand{\urlprefix}{URL }
\expandafter\ifx\csname urlstyle\endcsname\relax
  \providecommand{\doi}[1]{doi:\discretionary{}{}{}#1}\else
  \providecommand{\doi}[1]{doi:\discretionary{}{}{}\begingroup
  \urlstyle{rm}\url{#1}\endgroup}\fi
\providecommand{\bibinfo}[2]{#2}

\bibitem[{Babaioff et~al.(2017)Babaioff, Nisan, and
  Talgam-Cohen}]{Babaioff2017Competitive}
\bibinfo{author}{M.~Babaioff}, \bibinfo{author}{N.~Nisan},
  \bibinfo{author}{I.~Talgam-Cohen}, \bibinfo{title}{{Competitive Equilibria
  with Indivisible Goods and Generic Budgets}},
  \urlprefix\url{http://arxiv.org/abs/1703.08150}, \bibinfo{note}{arXiv
  preprint 1703.08150}, \bibinfo{year}{2017}.

\bibitem[{Barbanel(1996)}]{barbanel1996game}
\bibinfo{author}{J.~B. Barbanel}, \bibinfo{title}{Game-theoretic algorithms for
  fair and strongly fair cake division with entitlements}, in:
  \bibinfo{booktitle}{Colloquium Mathematicae}, vol.~\bibinfo{volume}{69},
  \bibinfo{pages}{59--73}, \bibinfo{year}{1996}.

\bibitem[{Brams(2007)}]{Brams2007Mathematics}
\bibinfo{author}{S.~J. Brams}, \bibinfo{title}{{Mathematics and Democracy:
  Designing Better Voting and Fair-Division Procedures}},
  \bibinfo{publisher}{Princeton University Press}, \bibinfo{edition}{1st} edn.,
  ISBN \bibinfo{isbn}{0691133212}, \bibinfo{year}{2007}.

\bibitem[{Brams et~al.(2011)Brams, Jones, and
  Klamler}]{Brams2011DivideandConquer}
\bibinfo{author}{S.~J. Brams}, \bibinfo{author}{M.~A. Jones},
  \bibinfo{author}{C.~Klamler}, \bibinfo{title}{{Divide-and-Conquer: A
  Proportional, Minimal-Envy Cake-Cutting Algorithm}}, \bibinfo{journal}{SIAM
  Rev.} \bibinfo{volume}{53}~(\bibinfo{number}{2}) (\bibinfo{year}{2011})
  \bibinfo{pages}{291--307}, \doi{\bibinfo{doi}{10.1137/080729475}},
  \urlprefix\url{http://drops.dagstuhl.de/volltexte/2007/1221/pdf/07261.BramsSteven.Paper.1221.pdf}.

\bibitem[{Brams and Klamler(2017)}]{Brams2017FairDivision}
\bibinfo{author}{S.~J. Brams}, \bibinfo{author}{C.~Klamler},
  \bibinfo{title}{{Fair Division}}, in: \bibinfo{editor}{R.~A. Meyers} (Ed.),
  \bibinfo{booktitle}{Encyclopedia of Complexity and Systems Science},
  \bibinfo{publisher}{Springer Berlin Heidelberg}, \bibinfo{address}{Berlin,
  Heidelberg}, ISBN \bibinfo{isbn}{978-3-642-27737-5}, \bibinfo{pages}{1--12},
  \bibinfo{year}{2017}.

\bibitem[{Brams and Taylor(1996)}]{Brams1996Fair}
\bibinfo{author}{S.~J. Brams}, \bibinfo{author}{A.~D. Taylor},
  \bibinfo{title}{{Fair Division: From Cake Cutting to Dispute Resolution}},
  \bibinfo{publisher}{Cambridge University Press}, \bibinfo{address}{Cambridge
  UK}, ISBN \bibinfo{isbn}{0521556449}, \bibinfo{year}{1996}.

\bibitem[{Br\^{a}nzei(2015)}]{Branzei2015Computational}
\bibinfo{author}{S.~Br\^{a}nzei}, \bibinfo{title}{{Computational Fair
  Division}}, Ph.D. thesis, \bibinfo{school}{Faculty of Science and Technology
  in Aarhus university}, \bibinfo{year}{2015}.

\bibitem[{Cseh and Fleiner(2018)}]{cseh2018complexity}
\bibinfo{author}{{\'A}.~Cseh}, \bibinfo{author}{T.~Fleiner},
  \bibinfo{title}{The complexity of cake cutting with unequal shares}, in:
  \bibinfo{booktitle}{Proceedings of SAGT'18},
  \bibinfo{organization}{Springer}, \bibinfo{pages}{19--30},
  \bibinfo{note}{arXiv preprint 1709.03152}, \bibinfo{year}{2018}.

\bibitem[{Farhadi et~al.(2017)Farhadi, Ghodsi, Hajiaghayi, Lahaie, Pennock,
  Seddighin, Seddighin, and Yami}]{Farhadi2017Fair}
\bibinfo{author}{A.~Farhadi}, \bibinfo{author}{M.~Ghodsi},
  \bibinfo{author}{M.~Hajiaghayi}, \bibinfo{author}{S.~Lahaie},
  \bibinfo{author}{D.~Pennock}, \bibinfo{author}{M.~Seddighin},
  \bibinfo{author}{S.~Seddighin}, \bibinfo{author}{H.~Yami},
  \bibinfo{title}{{Fair Allocation of Indivisible Goods to Asymmetric Agents}},
  in: \bibinfo{booktitle}{Proceedings of AAMAS'17},
  \urlprefix\url{http://arxiv.org/abs/1703.01649}, \bibinfo{note}{arXiv
  preprint 1703.01649}, \bibinfo{year}{2017}.

\bibitem[{Lindner and Rothe(2016)}]{Lindner2016}
\bibinfo{author}{C.~Lindner}, \bibinfo{author}{J.~Rothe},
  \bibinfo{title}{Cake-Cutting: Fair Division of Divisible Goods}, in:
  \bibinfo{editor}{J.~Rothe} (Ed.), \bibinfo{booktitle}{Economics and
  Computation: An Introduction to Algorithmic Game Theory, Computational Social
  Choice, and Fair Division}, \bibinfo{publisher}{Springer Berlin Heidelberg},
  \bibinfo{address}{Berlin, Heidelberg}, ISBN
  \bibinfo{isbn}{978-3-662-47904-9}, \bibinfo{pages}{395--491},
  \urlprefix\url{https://doi.org/10.1007/978-3-662-47904-9\_7},
  \bibinfo{year}{2016}.

\bibitem[{McAvaney et~al.(1992)McAvaney, Robertson, and
  Webb}]{McAvaney1992Ramsey}
\bibinfo{author}{K.~McAvaney}, \bibinfo{author}{J.~M. Robertson},
  \bibinfo{author}{W.~A. Webb}, \bibinfo{title}{{Ramsey partitions of integers
  and fair divisions}}, \bibinfo{journal}{Combinatorica}
  \bibinfo{volume}{12}~(\bibinfo{number}{2}) (\bibinfo{year}{1992})
  \bibinfo{pages}{193--201},
  \urlprefix\url{http://dx.doi.org/10.1007/bf01204722}.

\bibitem[{Moulin(2004)}]{Moulin2004Fair}
\bibinfo{author}{H.~Moulin}, \bibinfo{title}{{Fair Division and Collective
  Welfare}}, \bibinfo{publisher}{The MIT Press}, ISBN
  \bibinfo{isbn}{0262633116}, \bibinfo{year}{2004}.

\bibitem[{Procaccia(2015)}]{Procaccia2015Cake}
\bibinfo{author}{A.~D. Procaccia}, \bibinfo{title}{{Cake Cutting Algorithms}},
  in: \bibinfo{editor}{F.~Brandt}, \bibinfo{editor}{V.~Conitzer},
  \bibinfo{editor}{U.~Endriss}, \bibinfo{editor}{J.~Lang},
  \bibinfo{editor}{A.~D. Procaccia} (Eds.), \bibinfo{booktitle}{Handbook of
  Computational Social Choice}, chap.~\bibinfo{chapter}{13},
  \bibinfo{publisher}{Cambridge University Press}, \bibinfo{pages}{261--283},
  \urlprefix\url{http://procaccia.info/papers/cakechapter.pdf},
  \bibinfo{year}{2015}.

\bibitem[{Robertson and Webb(1997)}]{Robertson1997Extensions}
\bibinfo{author}{J.~M. Robertson}, \bibinfo{author}{W.~A. Webb},
  \bibinfo{title}{{Extensions of Cut-and-Choose Fair Division}},
  \bibinfo{journal}{Elemente der Mathematik}
  \bibinfo{volume}{52}~(\bibinfo{number}{1}) (\bibinfo{year}{1997})
  \bibinfo{pages}{23--30},
  \urlprefix\url{http://dx.doi.org/10.1007/s000170050007}.

\bibitem[{Robertson and Webb(1998)}]{Robertson1998CakeCutting}
\bibinfo{author}{J.~M. Robertson}, \bibinfo{author}{W.~A. Webb},
  \bibinfo{title}{{Cake-Cutting Algorithms: Be Fair if You Can}},
  \bibinfo{publisher}{A K Peters/CRC Press}, \bibinfo{edition}{first} edn.,
  ISBN \bibinfo{isbn}{1568810768}, \bibinfo{year}{1998}.

\bibitem[{Segal-Halevi(2017)}]{SegalHalevi2017Phd}
\bibinfo{author}{E.~Segal-Halevi}, \bibinfo{title}{{Fair Division of Land}},
  Ph.D. thesis, \bibinfo{school}{Bar Ilan University, Computer Science
  Department}, \bibinfo{note}{guided by Yonatan Aumann and Avinatan Hassidim},
  \bibinfo{year}{2017}.

\bibitem[{Segal-Halevi(2018{\natexlab{b}})}]{SegalHalevi2018CEFAI}
\bibinfo{author}{E.~Segal-Halevi}, \bibinfo{title}{{Competitive Equilibrium For
  almost All Incomes}}, in: \bibinfo{booktitle}{Proceedings of AAMAS'18},
  \bibinfo{note}{arXiv preprint 1705.04212},
  \bibinfo{year}{2018}{\natexlab{b}}.

\bibitem[{Segal-Halevi(2018{\natexlab{a}})}]{segalhalevi2018archipelago}
\bibinfo{author}{E.~Segal-Halevi}, \bibinfo{title}{Fair Division of an
  Archipelago}, \bibinfo{note}{arXiv preprint 1812.08150},
  \bibinfo{year}{2018}{\natexlab{a}}.

\bibitem[{Shishido and Zeng(1999)}]{Shishido1999MarkChooseCut}
\bibinfo{author}{H.~Shishido}, \bibinfo{author}{D.-Z. Zeng},
  \bibinfo{title}{{Mark-Choose-Cut Algorithms For Fair And Strongly Fair
  Division}}, \bibinfo{journal}{Group Decision and Negotiation}
  \bibinfo{volume}{8}~(\bibinfo{number}{2}) (\bibinfo{year}{1999})
  \bibinfo{pages}{125--137},
  \urlprefix\url{http://dx.doi.org/10.1023/a:1008620404353}.

\bibitem[{Steinhaus(1948)}]{Steinhaus1948Problem}
\bibinfo{author}{H.~Steinhaus}, \bibinfo{title}{{The problem of fair
  division}}, \bibinfo{journal}{Econometrica}
  \bibinfo{volume}{16}~(\bibinfo{number}{1}) (\bibinfo{year}{1948})
  \bibinfo{pages}{101--104},
  \urlprefix\url{http://www.jstor.org/stable/1914289}.

\bibitem[{Stromquist and Woodall(1985)}]{Stromquist1985Sets}
\bibinfo{author}{W.~Stromquist}, \bibinfo{author}{D.~R. Woodall},
  \bibinfo{title}{{Sets on which several measures agree}},
  \bibinfo{journal}{Journal of Mathematical Analysis and Applications}
  \bibinfo{volume}{108}~(\bibinfo{number}{1}) (\bibinfo{year}{1985})
  \bibinfo{pages}{241--248},
  \urlprefix\url{http://dx.doi.org/10.1016/0022-247x(85)90021-6}.

\end{thebibliography}
\end{document}